\documentclass[12pt,titlepage,a4paper,draft]{article}

\usepackage{color}

\definecolor{darkgreen}{rgb}{0,0.7,0}
 \definecolor{exp}{rgb}{0.85,0.7,0}
  \definecolor{msk}{rgb}{0.65,0.3,0.3}

\definecolor{msk2}{rgb}{0.75,0.3,0.1}
\definecolor{r}{rgb}{0.75,0.1,0.4}

\definecolor{exp2}{rgb}{0.85,0.6,0.1}

\usepackage{amssymb,latexsym,amsfonts,amsmath,amsthm}

\usepackage{url}
\usepackage{color}

\usepackage{verbatim}

{\makeatletter \@addtoreset{equation}{section}}

\newcommand{\N}{\mathbb{N}}

\newcommand{\Ra}{{\Rightarrow}}

\newcommand{\beq}{\begin{equation}}
\newcommand{\eeq}{\end{equation}}

\newcommand{\bml}{\begin{multline}}
\newcommand{\eml}{\end{multline}}

\newcommand{\bpm}{\begin{pmatrix}}
\newcommand{\epm}{\end{pmatrix}}

\newcommand{\bbm}{\begin{bmatrix}}
\newcommand{\ebm}{\end{bmatrix}}

\newcommand{\bcm}{}

\newtheorem{theorem}{Theorem}[section]
\newtheorem{lemma}[theorem]{Lemma}   

\usepackage{color}

\definecolor{ungu}{rgb}{0.8 , 0  , 0.7 }
 \definecolor{exp}{rgb}{0.85,0.7,0}
  \definecolor{msk}{rgb}{0.65,0.3,0.3}

\definecolor{msk2}{rgb}{0.75,0.3,0.1}
\definecolor{r}{rgb}{0.75,0.1,0.4}

\definecolor{exp2}{rgb}{0.85,0.6,0.1}

\definecolor{exp4}{rgb}{0.85,0.3,0.5}

\definecolor{vvs}{rgb}{0.5,0.4,0.7}

\begin{document}

\title{Bilinear characterizations of   companion matrices}

\author{ 
Minghua Lin\\ 
Department of Mathematics and Statistics\\
University of Victoria,  BC, 
\\Canada,  V8W 3R4 
\and
Harald~K. Wimmer
                       \\
Mathematisches Institut\\
Universit\"at W\"urzburg\\
97074 W\"urzburg, Germany\\
}
\date{\today}

\maketitle 

 \begin{abstract} 

Companion matrices of the second type are characterized by
properties that involve bilinear maps.


\tolerance 200

\vskip 2cm
\noindent
{\bf Corresponding Author:}\\
Harald K. Wimmer\\
Mathematisches Institut\\
Universit\"at W\"urzburg\\
97074 W\"urzburg, Germany
\flushleft{
{\textsf e-mail:}
~~\texttt{\small wimmer@mathematik.uni-wuerzburg.de} }
\flushleft{
{\textsf e-mail:} 
\texttt{  \small mlin87@ymail.com } 
 }

\end{abstract}

\section{Introduction} 

Let  $ K $ be a field.  The matrix 
\beq
 \label{eq.com}
  F_p =
\begin{bmatrix}
      0   &  .    &    &    &      &   p_0
\\
      1   &  0    & .   &   &      &   p_1
\\
          &   .   &  . & . &       &
\\
          &       &  . &  . &  .   &
\\
          &       &   & . &  .    &   .
\\
      0   &  0    &    &  &  1    &   p_{n-1}
\end{bmatrix}
\eeq
is the {\em{second companion matrix}}  
or (in the terminology of \cite{BTh}) 
 the {\em{companion matrix of the second  type}} 
associated with  
\[  p = [  p_0, p_1, \dots , p_{n-1} ] ^T  \in K^n . 
\]
It is well known (see e.g.\ \cite{DeT} for references)
 that companion matrices  are important 
in  linear algebra,
numerical analysis and applications,
e.g.\ in 
 systems and control theory and signal processing. 
In this paper we focus on  bilinear properties of companion matrices
that play a role  in the single-input case 
of  sensor-only fault detection  and identification \cite{B4},  \cite{Br}.

Let $ A  \in K^{ n \times n}   $, and 
$ b= [b_0, b_1, \dots , b_{n-1} ] ^T\in K ^n,  \:
g   
\in K ^n $. 
Set $ e_{n-1} = [0, \dots , 0 , 1] ^T \in K^n $ such that $F_p e_{n-1} = p $. 
The maps 
\[
 h: K^n \times K^n 
\to  K ^n ,  \:   
(b,g) \mapsto  h (A; b , g) =
[b, A, \dots , A ^{n-1} b] g  ,
\] 
and 
\begin{multline}   \label{eq.sbu} 
u :  K^n \times K^n \to K^n ,  \;  
 (b,g)
\mapsto  u (A; b , g)  = 
\\
\left[  
\begin{array}{l}
e_{n-1} ^T (b_0 I _n +  b_1 A +     b_2 A ^2 +   
 \cdots  
  b_{n-1} A^{n-1} )  
\\
e_{n-1} ^T (b_1 I _n +  b_2 A + 
 \cdots 
  b_{n-1} A^{n-2} )  
\\  \vdots 
\\
e_{n-1} ^T (  b_{n-2} I _n  +   b_{n-1} A  ) 
\\ e_{n-1} ^T (  b_{n-1} I _n )  
\end{array} 
\right] g 
\end{multline} 
are bilinear. 
It is the purpose of  this note  to
 show that a matrix $A$ is a second companion matrix 
if and only if the maps $ h (A;  b , g) $ and $ u(A;  b , g) $, respectively,
are symmetric.   
These results will be proved in  Section~\ref{sct.rslt}. 
In Section~\ref{sct.blck}  we deal with  block companion matrices and we   
describe extensions of results of  Section~\ref{sct.rslt} to matrix polynomials.

 \section{The main results}  \label{sct.rslt}

We shall use the following notation. 
The vectors  
\[
 e_0 = [1, 0 , \dots , 0 ] ^T , \,   \dots,  \,
e_{n-1} =  [0 , \dots , 0, 1] ^T ,  
\]
 are  the unit vectors of $K ^n $.
The characteristic polynomial of a matrix $ A  \in K^{ n \times n}   $
will be denoted by $ \chi_A(z) $.  
If $g \in K^n$  then the  Krylov matrix 
\[
 R( A , g) \,  = \, [ g, Ag, \dots , A ^{n-1} g] \, \in \, K^{ n \times  n}
\]
is the \!{\em{ reachability matrix}} (see e.g.\ \cite{Kai}) of the pair
\,$(A,g)$.    Note that  
\[
R (A, b) g  = h (A; b, g).
\] 
We recall \cite{Kai} that $A $ is similar to a companion matrix if and only if 
 $ R(A, g ) $ has full rank for some $g \in K^n$.
With  a vector $b   = [ b_0, b_1, \dots , b_{n-1}  ]^T  \in K^n $  
we associate the  polynomial  
  \,$b(z) =  
 b_0 + b_1 z + \cdots + b_{n-1} z ^{n-1} $.
Thus  $ b(A) = \sum _{i = 0 } ^{n-1} b_i A^i $. 
In particular, $e_0(z) = 1$, $\dots$, $ e_{n-1}(z) = z^{n-1}   $. 
Moreover,   if $ A \in K ^{n \times n} $ then
\beq \label{eq.rbao} 
R(A, g) b = b(A) g .
\eeq 
The following lemma {\rm{\cite{FW}}}   characterizes  companion matrices  in terms
of reachability matrices.  To make our note self-contained 
 we include a proof.

\begin{lemma}  \label{la.fab} 
For a matrix  $ A \in K^{ n \times n} $ 
the following statements are equivalent.
\begin{itemize} 
 \item[\rm{(i)}]  
 $ A = F_p $
for some $p \in K^n$. 
 \item[\rm{(ii)}] 
$  R ( A, g ) = g(A)$   
for  all $ g \in K^n$.
\item[\rm{(iii)}]  
$ R ( A, e_0) = e_0(A) = I _n$.
\end{itemize} 
\end{lemma}

\begin{proof}
It is obvious that $ A $ is a companion matrix of the
form  \eqref{eq.com} if and only if
\beq \label{eq.nm}
  A [ e_0,  e_1 , \dots , e_{n-2} ]
 =
   [ e_1,  e_2 , \dots , e_{n-1} ].
\eeq 
(i) $\Ra$ (ii):   We have to show that 
 \,$R ( F_p, g ) = g(F_p)$\,
holds
for all $g = \sum _{i = 0} ^{n-1} g_i e_i $. 
From    
\beq \label{eq.eifo} 
e_i = F_p ^i e_0 , \,\:  i = 0, \dots , n-1 ,
\eeq 
follows  \,$ R ( F_p , e_0 )  = I _n $.
Therefore 
\[
 R ( F_p , e_i ) = F_p ^i  \, R ( F_p , e_0 ) 
\]
implies  $   R ( F_p , e_i ) =  F_p ^i $.
Hence 
\[
    R ( F_p , g  ) = \sum \nolimits _{i = 0} ^{n-1}
g_i \,  R ( F_p , e_i  )
=
  \sum \nolimits _{i = 0} ^{n-1}  g_i  F_p ^i =
g(F_p).
\]
The implication 
(ii) $ \Ra $ (iii) is obvious.

(iii)  $ \Ra $ (i):  
From  $   R ( A, e_0 ) = I _n $   follows \eqref{eq.nm}. Therefore, 
$A$ is a companion matrix. 
\end{proof}

 It was shown in    \cite{Im} and  \cite[Proposition A.2]{B4} 
that matrices  in  second companion form satisfy an identity 
with ``curiously commuting vectors'',  stated  in   \eqref{eq.bca}  below. 
We note that the identity  \eqref{eq.bca}   is equivalent to the  
symmetry of the map $h(A; b,g) $.

\begin{theorem}  \label{cor.blin} 
Let $A \in K^{n \times n}  $.
We have $A = F_p $ for some $p \in K^n$ if and only if 
\beq  
\label{eq.bca} 
   [g, Ag, \dots , A^{n-1} g] b=   [b, Ab, \dots , A^{n-1} b] g . 
\eeq
\end{theorem}

\begin{proof}   
According to Lemma~\ref{la.fab} 
we have  $ A = F_p $ if and only if  
$ b(A)  = R(A, b) $  for all $b \in K^n $,
that is,  
\beq \label{eq.wmg}
 b(A) g  = R(A, b)g  \:\: \mbox{for all}\:\:   b , g \in K^n .
\eeq 
Then \eqref{eq.rbao} implies that  \eqref{eq.wmg} 
 is equivalent to 
\,$ R(A, b)g =   R(A, g) b$\, for all $  g, b \in K^n $.  
\end{proof}

We now deal with the map $   u(A; b, g) $.

\begin{theorem}  \label{thm.blw} 
Let $ A \in K ^{n \times n} $. The following statements are equivalent. 
\begin{itemize} 
 \item[\rm{(i)}]  $ A = F_p $ for some $ p \in K^n$.
\item[\rm{(ii)}] 
We have 
\beq \label{eq.jmtrs} 
 \bbm 
 e_{n-1} ^T   A^{n-1} 
\\
 e_{n-1} ^T   A^{n-2}  \\
\vdots 
\\
.
\\
e_{n-1} ^T   A
\\
e_{n-1} ^T 
\ebm = 
\left[ 
\begin{array}{lcllll}
1  &   e_{n-1} ^T A e_{n-1} &   & \dots &   &   e_{n-1} ^T A^{n-1}  e_{n-1} 
\\ 
&  1   &  
 & \dots &   &  e_{n-1} ^T A^{n-2}  e_{n-1}  \\
  &  & \ddots  &  & &   \vdots 
\\
 &&    &  &  &   
\\
& &   &    &  1   &  e_{n-1} ^T A  e_{n-1}  
 \\
&&    &  &  &    1
\end{array} 
\right] .
\eeq  
  \item[\rm{(iii)}]  The bilinear map $   u(A; b, g) $ in \eqref{eq.sbu}  
satisfies 
$   u(A; b, g) =   u(A; g, b) $\,  for all  
$ b , g \in K^n   $.
\end{itemize}
\end{theorem}

\begin{proof} 
Define 
\[ 
L(A, g) = 
\begin{bmatrix} 
e_{n-1} ^T g  &   e_{n-1} ^T A g  & \dots &   &   e_{n-1} ^T A^{n-1}  g
\\
&  e_{n-1} ^T g  &  e_{n-1} ^T A g  & \dots &    e_{n-1} ^T A^{n-2}  g
\\
&   &\ddots  & \ddots &   \vdots 
\\
&   &    &  e_{n-1} ^T g  &  e_{n-1} ^T A  g 
 \\
&   &  &  &    e_{n-1} ^T  g
 \end{bmatrix} 
\]
and 
\[
Q (A , g ) = 
\left[  
\begin{array}{l}
e_{n-1} ^T (g_0 I _n +  g_1 A +     g_2 A ^2 +   
 \cdots  
  g_{n-1} A^{n-1} )  
\\
e_{n-1} ^T (g_1 I_n  +  g_2 A + 
 \cdots 
  g_{n-1} A^{n-2} )  
\\  \vdots 
\\
e_{n-1} ^T (  g_{n-2} I _n  +   g_{n-1} A  ) 
\\ e_{n-1} ^T (  g_{n-1} I _n )  
\end{array} 
\right]  .
\] 
Then $  u(A; b, g) $ in  \eqref{eq.sbu}  can be written as 
\[  u(A; b, g)  = L(A, g) b .
\]
We have   $  u(A; g, b) =  Q ( A , g) b $.
Hence (iii) holds if and only if 
\beq \label{eq.jbpa} 
 L(A, g) =   Q ( A , g) \ \mbox{for all} \ g \in K^n. 
\eeq
The maps $ g \mapsto  L(A, g) $ and $ g \mapsto   Q ( A , g)  $
are linear. 
Hence \eqref{eq.jbpa}  is equivalent to
\beq \label{eq.font} 
 L(A, e_{i}) =   Q ( A , e_i) ,  \: i = 0, \dots , n-1 . 
\eeq 
In the case $ i = n-1 $ the corresponding matrices are 
\[
Q ( A , e_{n-1} ) = 
\begin{bmatrix} 
 e_{n-1} ^T A^{n-1}  
\\
  e_{n-1} ^T A^{n-2}  
\\
 \vdots 
\\
  e_{n-1} ^T A  
 \\
  e_{n-1} ^T 
 \end{bmatrix}
\]
and
\[
 L(A, e_{n-1}) = 
\begin{bmatrix} 
1  &   e_{n-1} ^T A e_{n-1}   & \dots &   &   e_{n-1} ^T A^{n-1}  e_{n-1} 
\\
&  1   &  e_{n-1} ^T A e_{n-1}   & \dots &    e_{n-1} ^T A^{n-2}  e_{n-1} 
\\
&   &\ddots  & \ddots &   \vdots 
\\
&   &    &  1   &  e_{n-1} ^T A  e_{n-1}  
 \\
&   &  &  &    1
 \end{bmatrix} .
\] 

(i)  $ \Ra$ (iii): 
Suppose $ A = F_p $. 
Then $ A ^s e_ i = A^{s + i } e_0 $ if $ 0 \le i \le n -1$ and $ s \in \N_0$. 
We have 
\[
L ( A , g) e_r =  
[ e_{n-1} ^T A^r g ,  \dots ,  e_{n-1} ^T A g, 
e_{n-1} ^T g,  0, \dots , 0 ]^T. 
\]
Thus 
\[
L ( A , e_i ) e_r =   
[ e_{n-1} ^T  A^{r + i} e_0  ,  \dots ,  e_{n-1} ^T A^{i +1} e_0 , 
e_{n-1} ^T A^i e_0 ,  0, \dots , 0 ] ^T. 
\] 
From  
\[
Q ( A , e_i )   =  \bbm  e_{n-1} ^T  A^i \\ 
\vdots\\    e_{n-1} ^T  A\\
  e_{n-1} ^T \\ 0 \\ \vdots \\ 0
\ebm 
\]
we obtain 
$Q ( A , e_i ) e_r = L ( A , e_i ) e_r  $, $ r = 0 , \dots ,  n-1$. 
Hence \eqref{eq.font} is satisfied, which is equivalent to (iii).

(iii) $ \Ra$ (ii):  
 We have seen that (iii) is equivalent to 
\eqref{eq.font}.   Choosing  $ i = n-1$ in \eqref{eq.font}
yields $Q(A, e_{n-1}) = L(A, e_{n-1}) $, that is \eqref{eq.jmtrs}.   

(ii) $\Ra$ (i):  Suppose \eqref{eq.jmtrs} holds.   Let 
\[
A^k = \Big[ a_{ij} ^{(k)}   \Big]_{i,j= 0}^{n-1} \, ,  \:  \: k = 0, 1, \dots, n-1. 
\]  
We show that the rows of $A$ are the rows of a companion  matrix.
The proof is by induction.  The induction hypothesis is  
\begin{multline}  \label{eq.ink} 
  e_{n -\nu}^T  A = 
e_{n-\nu -1} ^T + a_{n-\nu, n-1}  e_{n-1} ^T  = [ 0 , \dots , 0 ,1, 0, \dots , 0, a_{n-\nu, n-1} ] ,
\\
1 \le \nu < n-1 . 
\end{multline}
From \eqref{eq.jmtrs} we obtain 
\[   e_{n- 1}^T  A  = [ 0 , \dots , 0 ,1, e_{n-1 } ^T A e_{ n-1} ] 
=  [ 0 , \dots , 0 ,1,  a_{n-1 , n-1} ]  . 
\]
Hence  \eqref{eq.ink} is satisfied for $ \nu = 1$. 
Assume that \eqref{eq.ink}  is valid for $ \nu = 1, \dots , k $. 
Then
$ 
   A =  \bbm   W_k \\ H_k \ebm  $ 
with  
\beq \label{eq.hwg} 
\bbm e_{n- k} ^T  \\ \vdots \\ e_{n-2}  ^T \\ e_{n-1}  ^T  \ebm  A 
= H_k =  \bbm
0   & \dots  &  0 & 1 & 0 & \dots & 0 &  a_{ n -k ,  n-1 } 
\\
&   &  & & \ddots     & &  & 
\\ 
0   & \dots  &  0  & 0 & \dots &1 & 0 &  a_{n-2 ,  n-1 } 
\\
0   & \dots  & 0   & 0 & \dots &0 & 1 &  a_{n-1  ,  n-1 } 
\\ 
\ebm _{k \times n} .
\eeq  
Set 
\[ e_{n-1 - k } ^T A = [ \tilde a_0,   \,   \dots ,   \,   \tilde a_{n-2} , \,  \tilde a_{n-1} ].
\]
From
 \eqref{eq.jmtrs}   and
 \eqref{eq.hwg} 
follows 
\begin{multline*}  
e_{n-1} ^T A ^{k+1} =  
 ( e_{n-1} ^T A ^{k} ) A = 
[0, \dots , 0, 1, 
a_{n-1, n-1} ,   a_{n-1, n-1} ^{(2)} , \dots , a_{n-1, n-1} ^{(k)} ]  
\\
\bbm 
. & .  & \dots  & .  &  .  & \dots  &   . & .
\\
\tilde a_0
.  & \dots  & \tilde a_{ n - k -2}  &   \tilde a_{ n-k -1 } 
& . & \dots &   \tilde a_{n- 2} 
 &  \tilde a_{n-1 }
\\
0 &  \dots & 0                                 &       1 
&      0   & \dots  &   0 
 &  a_{n- k, n-1}
\\
0 & \dots  & .  & 0 & 1  & \dots  &   0 &   a_{n-  k + 1 , n-1}
\\
 . &  \dots  & . & . & .  & \dots    & .  & . 
\\
0 & \dots   &  & 0  & \dots   &  1 & 0   & a_{n-  2 , n-1}
\\
0 & \dots   & 0 &  0  &\dots   &  0 &  1 & a_{n-  1, n-1}
 \ebm  = 
\\
[ \tilde a_{ 0} , \tilde a_1  \dots,  \tilde  a_{ n - k - 2} ,   \, 
 \tilde a_{ n - k  -1}  +   a_{ n-1 , n-  1},   
\\
 \tilde a_{ n - k }  +   a_{ n-1 , n-  1}  ^{(2)} , \dots , 
 \tilde a_{n-2} +  a_{ n-1 , n-  1}  ^{(k)} ,   \,  a_{ n-  1, n-1}  ^{(k+1)}    ] .
\end{multline*} 
The assumption \eqref{eq.jmtrs} implies
\[
e_{n-1} ^T A ^{k+1} 
=  
  [0, \dots , 0,  1, 
  a_{ n-1, n-1} ,  a_{ n-1, n-1}   ^{(2)} , \dots ,   a_{ n-1 , n-  1}  ^{(k)},
 a_{ n-1 , n-  1}  ^{(k+1)} ]. 
\]
Hence  
$ [ \tilde a_{ 0} , \dots  ,  \tilde a_{ n - k - 3} ] = [0, \dots , 0] $, and \,
$ \tilde  a_{ n - k - 2} = 1  $.
Moreover
\begin{multline*}
[\tilde a_{n - k  -1}  +   a_{ n-1 , n-  1},   
\tilde a_{ n - k }  +   a_{ n-1 , n-  1}  ^{(2)} , \dots , 
\tilde  a_{n-2} +  a_{ n-1 , n-  1}  ^{(k)} ]
=
\\
[
 a_{ n-1, n-1} , 
 a_{ n-1, n-1}   ^{(2)} , \dots ,   a_{ n-1 , n-  1}  ^{(k)} ]
\end{multline*}
yields  
\,$ [ \tilde  a_{ n - k  -1} ,  \tilde   a_{n - k }  , \dots , 
\tilde a_{n-2} ] = [0, \dots , 0]  $.
This proves  \eqref{eq.ink}  in the case $ \nu = k+1$, and we obtain 
$A = H_n = F_a $ with
$ a = [a_{0,n-1} , \,  \dots , \,  a_{n-1,n-1} ]^T$. 
\end{proof}

It was proved in  \cite{B4} that the map  $( g, b)  \mapsto u(A; g, b) $ 
is symmetric if  $A = F_p $. 
The proof of  \cite[Proposition A.5]{B4} 
with the  
 identities
\[
e_{n-1} ^T  \left(\sum _{j = k} ^{n-1} b_j A^{j- k } \right) g 
=
e_{n-1} ^T \left(\sum _{j = k} ^{n-1} g_j A^{j- k } \right) b ,
\:\, k=0, 1, \dots , n-1, 
\] 
is rather involved. 
We remark  that 
\[
L(F_p, e_{n-1} ) = 
\bbm   1 & p_{n-1} & p_{n-2} & \dots & p_1 
\\
 & 1 & p_{n-1}   & \dots & p_2
\\
& & \ddots   & \ddots &  
\\
&  &  &.  &  .
\\
&  &  &  & 1 
\ebm  .
\]

The ``crossover'' identity   \eqref{eq.cim} 
below appears in 
 \cite[Proposition A.3]{B4}.
\begin{theorem}  \label{thm.cr}    
Suppose  $ A \in K^{n \times n} $ and 
\beq \label{eq.chpl} 
 \chi _A (z) = z^n - \sum   \nolimits _{i=0} ^{n-1} p_i z^i =
 z^n  - p(z) . 
\eeq
Let 
$g = [ g_0, \dots , g_{n-1} ]  ^T \in K^n $, 
 $b= [ b_0, \dots , b_{n-1} ]  ^T \in K^n$, 
and $g_n, b_n \in K$.
 The following statements are equivalent.
\begin{itemize} 
  \item[\rm{(i)}]  
The identity 
\beq \label{eq.cim} 
\sum \nolimits _{k=0} ^n   A ^k g_{k} 
( b  +  b_n  p   ) =
\sum \nolimits _{k=0} ^n   A ^k  b_{k}
( g + g_n  p   ) 
\eeq
holds 
for all   $ b, b_n , g , g_n$. 
  \item[\rm{(ii)}]  
The matrix $A$ is a second companion matrix such that 
$ A = F_p $. 
\end{itemize} 
\end{theorem}

\begin{proof} (i) $\Rightarrow$ (ii): 
Note that   \eqref{eq.cim}  is equivalent to
\beq \label{eq.nomw} 
R(A,  b ) g  +   b_n R(A,      p )g   + g_n  A^n b  
= R( A , g  ) b + g_n R( A ,  p ) b +  b_n  A^n g  .
\eeq 
If  we choose 
$b_{n} = g_{n} = 0$ in 
 \eqref{eq.nomw}   then we obtain 
$  R  ( A ,  b) g =  R ( A ,  g) b $.  
Hence   $ A = F_q$ for some $q \in K^n$
(by  Theorem~\ref{cor.blin}). 
Then \eqref{eq.chpl} implies  $q = p$.

(ii) $\Ra$ (i): 
Assume $ A =F_p$. 
Then  
$ R(A,  b ) g = R( A , g  ) b $ for all $ b, g \in K^n$. 
We also have     $A^n = p(A)$  and therefore  $A^n g =  p(A) g = R(A, p ) g $.
Hence \eqref{eq.nomw} 
holds for all $b,g, b_n , g_n$. 
\end{proof} 

\section{Block companion matrices}   \label{sct.blck} 

In this section 
we extend results of the preceding section
to block companion matrices. 
Let 
$ P_i \in K^{t \times t} $, $ i = 0, \dots , n-1 $,
be the entries of the block matrix   
\[
P = \bbm P_0 \\ \vdots \\ P_{n-1} \ebm  _{nt \times t},
\] 
and    the coefficients of the monic matrix polynomial  
\[
\tilde P (z)  = z^n I_t  -(  z^{n-1}  P_{n-1}  + \dots + P_0)  . 
\]
A linearization \cite{GLR}  of  $ \tilde P(z) $ gives rise to
the {\em{block companion matrix of the second type}} 
\[
  F_P=
\begin{bmatrix}
      0   &  .    &    &    &      &   P_0
\\
      I_t  &  0    & .   &   &      &   P_1
\\
          &   .   &  . & . &       &
\\
          &       &  . &  . &  .   &
\\
          &       &   & . &  .    &   .
\\
      0   &  0    &    &  &  I _t   &   P_{n-1}
\end{bmatrix}
\in K^{nt \times nt}  .
\]
Our main tool is the Kronecker product.  
Let  $ S = ( s_{ij}) \in K ^{ \ell \times m } $
and  $ T \in K^{p \times r} $ then   
$ S \otimes T = ( s_{ij} T ) \in K^{\ell p \times m r}  $. 
Define 
\[
E_0 = e_0 \otimes I_t = [I_t, 0 , \dots , 0 ] ^T,  
\dots,  E_{n-1} = e_{n-1} \otimes I_t = [0, \dots , 0, I_t ] ^T.
\]
To the block matrix 
\[
B = \begin{bmatrix}  B_0 \\ B_1 \\ \vdots \\ B_{n-1}  
\end{bmatrix} = \sum \nolimits _{i= 0} ^{n-1} E_i B_i 
 \in K ^{nt \times  m }
\]
with  $ B_j \in K ^{t \times m }$, $ j = 0, \dots , n-1$, 
 we associate the matrix polynomial
\[
B(z)  = B_0 + z B_1 + \cdots +  z^{n-1} B_{n -1} \in K^{t \times m}[z] . 
\] 
Let $A \in K ^{nt \times nt } $.  We define 
\beq \label{eq.hts} 
B(A) =  ( I _n \otimes B_0 ) + A  ( I _n \otimes B_1)
 +
\cdots + A^{n-1}( I _n \otimes B_{n -1}) \in K^{nt \times n m}.
\eeq
In \eqref{eq.hts} 
we have 
 an example of an operator substitution
(see  \cite{Hau}). 
The matrix 
\[
R(A, B) = [B, AB, \dots ,  A^{n-1} B ] \in K ^{nt \times nm} 
\] 
is the $n$-step reachability matrix of the pair $(A,B) $. 
 Let  $B_i,  G_i \in K^{t \times t} $, $i = 0, \dots , n-1$.
We say that the $ nt \times  t $ matrices   
\[
B = \begin{bmatrix}  B_0 \\ B_1 \\ \vdots \\ B_{n-1} 
\end{bmatrix} 
\quad  {\rm{and}} \quad 
G  = \begin{bmatrix}  G_0 \\ G_1 \\ \vdots \\ G_{n-1} 
\end{bmatrix} 
\]
are {\em{blockwise commuting}} if 
 \,$B_i G_j = G_j B_i $,  $i, j = 0, \dots , n-1$. 
Adapting arguments of the preceding section 
to  Kronecker products   we  prove the following  result.

\begin{theorem} \label{thm.chrr} 
Let $ A \in K^{nt \times nt}$. 
The following statements are equivalent. 
\begin{itemize} 
\item[\rm{(i)}] 
$ A = F_P $ for some $P \in K^{nt \times t }   $. 
\item[\rm{(ii)}]  
$ R(A,E_0) = I_{nt} $.  
\item[\rm{(iii)}] 
$R(A,G) =  G(A) \:\: for \;\; all  \:\: G \in  K ^{nt \times m }$.
\item[\rm{(iv)}] 
The identity  
\beq    \label{eq.crg} 
[ B , AB , \dots , A^{n-1}  B ] G
= 
[ G , AG , \dots , A^{n-1} G ] B 
\eeq 
holds 
for all blockwise commuting matrices $B,G \in K^{nt \times t }$.
\end{itemize} 
\end{theorem}

\begin{proof} 
We proceed along the lines suggested by a referee. 
\begin{description} 
\item[\rm{(ii) $\Ra$ (i):}] 
The implication is obvious. 
\item[\rm{(i) $\Ra$ (iii):}] 
For $ A \in K^{nt \times nt} $ we have 
\[
R (A, G ) = \sum \nolimits _{i = 0 } ^{n-1}  R (A, E_i ) ( I_n  \otimes G_i ) .
\]
If $ A = F_P $ then $  R (A, E_i )  =  F_P ^i $, and we obtain $ R (A, G ) = G(A)$.
\item[\rm{(iii) $\Ra$ (iv):}] 
If  $B$ and $G$ are blockwise commuting then 
\[ 
B(A) G =  \sum \nolimits_{i , j = 0} ^{n-1} 
 A ^i ( I_n  \otimes G_i ) ( I_n \otimes B_ j)  =  G(A) B.
\]
Thus  (iii) implies 
$ R(A, B) G = R(A, G) B$, which is the target identity \eqref{eq.crg}. 
\item[\rm{(iv) $\Ra$ (ii):}]  
Choose  $G= E_0 $ in \eqref{eq.crg}.  Then it  follows from  $ R(A, B)  E_0 = B$  
  that $ R(A, E_0) B = B $  holds for all $ B \in K ^{nt \times t}$.
Hence $ R(A, E_0) = I _{nt} $. 
\end{description} 
\end{proof}

{\bf{Acknowledgement.}} 
We are grateful to a referee for detailed comments and 
  valuable  suggestions,  which helped us to improve the paper significantly.


\begin{thebibliography}{99}



\bibitem{BTh}  H. Bart and  G. Ph. A. Thijsse, 
Simultaneous reduction to companion and triangular forms  of sets of 
matrices,  Linear Multilinear Algebra   26, 231--241 (1990).
 


\bibitem{Br} 
A. J. Brzezinski, 
Output-Only Techniques for Fault Detection,
Dissertation, Department of Aerospace Engineering,
University of Michigan,  Ann Arbor, 2011. 

\bibitem{Im} 
 A. J. Brzezinski, E. Wu, and   D. S.  Bernstein,
Curiously commuting vectors, Problem 44-3,  
IMAGE (Bulletin of the International Linear 
Algebra Society) 44 (2010).  




\bibitem{B4} 
A. J. Brzezinski, S. Kukreja,   Jun Ni, and   D. S. Bernstein, 
  Sensor-only fault detection using pseudo transfer function 
identification,  in 
 Proc. Amer. Contr. Conf.,    5433--5438,
Baltimore,  June 2010.


\bibitem{DeT} 
F. De Ter\'{a}n,  F.  M.  Dopico, and J.  P\'{e}rez,
Condition numbers for inversion of Fiedler matrices,   
Linear Algebra Appl. 439, 944--981 (2013).


\bibitem{FW} 
A. Ferrante and H. K. Wimmer, 
Reachability matrices and cyclic matrices,
Electron. J. Linear Algebra 20, 95--102 (2010).



\bibitem{GLR} 
   I. Gohberg, P. Lancaster, and L. Rodman,
Matrix Polynomials,  Academic Press, New York, 1982. 


\bibitem{Hau}
M. L. J. Hautus, 
Operator substitution, Linear Algebra  Appl.  205--206,  713--739 (1994).

\bibitem{Kai} 
Th. Kailath, Linear Systems,   Prentice Hall,  Englewood Cliffs,   1980.


\end{thebibliography}
\end{document}